\documentclass[12pt]{amsart}

\usepackage[margin=1.5in]{geometry} 

\usepackage{amsfonts}
\usepackage{amssymb}
\usepackage{amsthm}
\usepackage{amsmath}
\usepackage{enumerate}
\usepackage{hyperref}
\hypersetup{
    colorlinks=false,
    pdfborder={0 0 0},
}
\usepackage{amsrefs}
\usepackage{graphicx}
\usepackage{verbatim}
\usepackage{amscd}

\usepackage{tikz}
\usetikzlibrary{matrix,arrows,decorations.pathmorphing}

\theoremstyle{plain}

\newcommand{\IC}{\mathbb{C}}
\newcommand{\cexp}{\mathcal{E}}

\newcommand{\norm}[1]{\left\Vert#1\right\Vert}
\newcommand{\set}[1]{\left\{#1\right\}}

\newtheorem{proposition}{Proposition}[section]
\newtheorem{lemma}[proposition]{Lemma}
\newtheorem{theorem}[proposition]{Theorem}
\newtheorem{corollary}[proposition]{Corollary}

\theoremstyle{definition}
\newtheorem{definition}[proposition]{Definition}
\theoremstyle{remark}

\newtheorem*{convention*}{Convention}
\newtheorem{example}[proposition]{Example}

\newtheorem{remark}[proposition]{Remark}

\numberwithin{equation}{section}

\newcommand{\ran}{\hbox{ran\,}}

\title[Homomorphic conditional expectations]{
Homomorphic conditional expectations \\ as noncommutative retractions}

\author{Robert Pluta}
\address{Department of Mathematics, University of California, Irvine, CA 92697-3875, USA}
\email{plutar@tcd.ie}

\author{Bernard Russo}
\address{Department of Mathematics, University of California, Irvine, CA 92697-3875, USA}
\email{brusso@math.uci.edu}

\date{\today}

\begin{document}

\begin{abstract}
   Let $A$ be a $C^*$-algebra and $\cexp\colon A \to A$ a conditional expectation.
  The Kadison-Schwarz inequality for completely positive maps, $$\cexp(x)^* \cexp(x) \leq \cexp(x^* x),$$
  implies that
  $$
  \norm{\cexp(x)}^2 \leq \norm{\cexp(x^* x)}.$$
  In this note we show that $\cexp$ is homomorphic
  (in the sense that $\cexp(xy) = \cexp(x)\cexp(y)$ for every $x, y$ in $A$)
  if and only if
  $$\norm{\cexp(x)}^2 = \norm{\cexp(x^*x)},$$
for every $x$ in $A$.  
  We also prove that a homomorphic conditional expectation on a commutative $C^*$-algebra $C_0(X)$ is given by composition with a continuous retraction of $X$.
 One may therefore consider homomorphic conditional expectations
  as noncommutative retractions.
  \end{abstract}
\maketitle
  \section{Introduction}
  It is easy to see that a conditional expectation $\cexp$ is homomorphic
  if and only if the kernel of $\cexp$ is an ideal.
  Thus, there are no nontrivial homomorphic conditional expectations
  on  simple $C^*$-algebras,
  but it makes sense to study homomorphic conditional expectations
  on $C^*$-algebras with rich ideal structure.
  It follows from \cite[Theorem 3.1]{Choi74} that a conditional expectation is homomorphic if and only if equality holds in the Kadison-Schwarz inequality for every $x$.  In our main result, Theorem~\ref{t:CharacterizationHomomorphicExpectations} below, we weaken the latter condition to equality of the norms.

  A central projection $p$ in a $C^*$-algebra $A$
  gives rise to a homomorphic conditional expectation
  $\cexp_p \colon A \to A$ given by $\cexp_p(x) = px$ for all $x$ in $A$.  As a bi-product of our main result, we prove a converse in Corollary~\ref{cor:0126171}.
 
A retraction  of a locally  compact Hausdorff space $X$, that is, a continuous map $\tau:X\rightarrow X$ such that $\tau\circ \tau=\tau$,
  gives rise to a homomorphic conditional expectation
  $\cexp_\tau \colon C_0(X) \to C_0(X)$ given by
  $\cexp_\tau(f) = f\circ \tau$ for all functions $f$ in $C_0(X)$.
  There are expectations on $C(K)$, $K$ compact, which do not come from retractions of $K$,
  but those expectations are not homomorphic.
  A~unital conditional expectation $\cexp \colon C(K) \to C(K)$
  is homomorphic if and only if it comes from some retraction of $K$ (Theorem~\ref{t:Retraction=HomomorphicExpectation} below),
  and, in accordance with Theorem~\ref{t:CharacterizationHomomorphicExpectations} below, this in turn is equivalent to the requirement that the conditional expectation satisfies
  $\norm{\cexp(f)}^2 = \norm{\cexp(|f|^2)}$
  for every $f$ in $C(K)$.
  A similar result holds for (not necessarily unital) commutative $C^*$-algebras  $C_0(X)$
  for a locally compact Hausdorff space $X$.
  Thus, in the framework of Gelfand duality,
  we have the equivalence:
  \begin{eqnarray*}
  \left(
  \begin{array}{c}
  \mbox{Retractions $\tau \colon X \to X$ of } \\
  \mbox{locally compact spaces $X$ } \\
  \end{array}
  \right) &   \Leftrightarrow   & \left(
  \begin{array}{c}
  \mbox{Homomorphic   }  \\
   \mbox{conditional expectations  }  \\
  \mbox{ $\cexp\colon C_0(X) \to C_0(X)$    } \\
  \end{array}
  \right) \\
   & \Leftrightarrow& \left(
  \begin{array}{c}
  \mbox{Conditional expectations}  \\
  \mbox{ $\cexp\colon C_0(X) \to C_0(X)$    } \\
  \mbox{with  $\norm{\cexp(f)}^2 = \norm{\cexp(|f|^2)}$ } \\
  \end{array}
  \right) \\
  \end{eqnarray*}
  We believe that this justifies the following definition: 
  A {\em noncommutative retraction} on a $C^*$-algebra $A$
  is a conditional expectation $\cexp \colon A \to A$ with
  $\cexp(xy) =\cexp(x) \cexp(y)$ for all $x, y\in A$.
  (By Theorem~\ref{t:CharacterizationHomomorphicExpectations} below, this is equivalent to the requirement that the conditional expectation satisfies
  $\norm{\cexp(x)}^2 = \norm{\cexp(x^*x)} $ for $x\in A$.)

\section{Basic properties of conditional expectations}

In this section we review some basic facts and terminology that relate to conditional expectations
in a general noncommutative setting of 
$C^*$-algebras.

\begin{definition}
\label{d:ConditionalExpectation}
A {\em conditional expectation}
defined on a $C^*$-algebra $A$
is a positive linear map    $\cexp \colon A \to A$
satisfying  $\cexp^2 =\cexp$ (where  $\cexp^2 = \cexp \circ \cexp$) and
\[
\cexp(\cexp(x) y) = \cexp(x) \cexp(y) \quad \mbox{for every $x, y$ in $A$.}
\]
\end{definition}

It follows that the range of $\cexp$ is a $C^*$-subalgebra of $A$. A  conditional expectation  $\cexp \colon A \to A$  also satisfies
\[
\cexp(x \cexp(y)) = \cexp(x) \cexp(y) \quad \mbox{for every $x, y$ in $A$.}
\]
  Thus $\cexp$ is a bimodule map over~its~range.  Moreover, $\cexp$ is completely positive and has norm 1  (\cite[Corollary II.6.10.3]{Blackadar2006}).
  The Kadison-Choi-Schwarz inequality is proved in \cite[Corollary 2.8]{Choi74}.

If $A$ is unital with the identity element $1$,
then the projection $e = \cexp(1)$ is an identity element of the range,
which is contained
in the corner $eAe$, i.e., the largest $C^*$-subalgebra of $A$ containing $e$ as the identity element.

\begin{remark}
By a corner of a $C^*$-algebra $A$ we mean a $C^*$-subalgebra $S$ of $A$
with the additional property that there is a norm closed linear subspace $M$ of $A$
such that  
$A = S \oplus M$
and $M$ is invariant under both left and right multiplication by elements of $S$, i.e. $SM \subseteq M$, $MS \subseteq M$.
It follows automatically that $M$ is also invariant under the $*$-operation, i.e. $M^* = M$, 
so it can be regarded as a (not necesarily unital) involutive Banach $S$-bimodule.
If $\cexp \colon A \to A$ is a conditional expectation, then the range of $\cexp$ is a corner of $A$.
On the other hand, if a $C^*$-subalgebra $S$ is not a corner of $A$, then there is no conditional expectation from $A$ onto $S$.
\end{remark}

It is clear that a corner of a unital $C^*$-algebra must be unital,
however the identity element of the corner need not be the same as the identity element of the ambient $C^*$-algebra.
This observation is useful.
It shows, for example, that if $H$ is an infinite-dimensional Hilbert space,
then the algebra of compact operators $\mathcal{K}(H)$ is not a corner of $\mathcal{B}(H)$.
Consequently, there is no conditional expectation from $\mathcal{B}(H)$ onto $\mathcal{K}(H)$.
Exactly the same argument shows that
there is no conditional expectation from $\ell^\infty$  onto  $c_0$.
Of course, these two observations can be strengthened to the assertion that
there is no conditional expectation from a unital $C^*$-algebra $A$ onto
a non-unital  $C^*$-subalgebra of $A$.

Regarding terminology,
we will occasionally refer to a conditional expectation
simply as an expectation
leaving the word  ``conditional''  implicit.
The following remark
provides us with some basic properties of expectations.

\begin{remark}
Let $A$ be a $C^*$-algebra
and let $\cexp\colon A \to A$ be a conditional expectation.
The range of $\cexp$,
which we denote by $S$,
is the $C^*$-subalgebra of $A$ consisting of all fixed points of $\cexp$.
The kernel of $\cexp$
is a norm closed linear subspace of $A$
that is closed under the $*$-operation and
invariant under left and right multiplication by elements of $S$.
In particular, letting $M$ be the~kernel~of~$\cexp$, one has
$M^* = M$,  $SM\subseteq M$,  $MS \subseteq M$.
The space $M$ can be regarded as an involutive Banach $S$-bimodule
if one does not require that $1m = m1 = m$,  for all $m\in M$,
even if $S$ has an identity $1 = 1_S$.
With this convention,
$A = S \oplus M$
is a direct sum in the category of involutive Banach $S$-bimodules
and the following  sequence
$$
\begin{CD}
0   @>>>   S   @>1>>   A   @>1-\cexp>>   M   @>>>   0
\end{CD}
$$
of $*$-preserving $S$-bimodule maps is exact.
\end{remark}

There is a link between certain projections and expectations.
It has already been observed that every nonzero
conditional expectation $\cexp\colon A \to A$ defined on a $C^*$-algebra $A$
is a projection of norm one. 
The converse of this observation does not hold in general.
For example, the mapping from the matrix algebra $M_2(\IC)$ into itself
that replaces each main diagonal entry of every 2-by-2 matrix with zero
is a projection of norm one, yet it is not a conditional expectation
because its range is not a subalgebra of $M_2(\IC)$.
However, every projection of norm one
whose range is a subalgebra
must be a conditional expectation;
this is a general version of the
well known theorem of Tomiyama~\cite{Tomiyama1957}, which we mention
for the sake of completeness. 
\section{Homomorphic conditional expectations}

The main result of this section is
Theorem~\ref{t:CharacterizationHomomorphicExpectations}.

\begin{definition}
Let $A$ be a $C^*$-algebra.
A conditional expectation $\cexp \colon A \to A$
is  {\em homomorphic} (or {\em multiplicative})
if
$
\cexp(xy) = \cexp(x)\cexp(y)
$
for every $x, y$ in $A$.
\end{definition}

We now give some examples of homomorphic  conditional expectations.
The first one describes a connection between
homomorphic conditional expectations and $C^*$-algebra homomorphisms.

\begin{example}
[Expectations onto graphs of $C^*$-algebra homomorphisms]
Let $A, B$ be $C^*$-algebras.
Let $A\oplus B$ be the $C^*$-algebra endowed with the maximum norm, 
with the summands as ideals,
and the algebraic operations performed pointwise.
If $\phi \colon A \to B$ is a $*$-homomorphism,
then the map
\begin{equation}
\label{e:CexpOntoGraph}
\cexp \colon A\oplus B \to A\oplus B,
\qquad \cexp(x , y) =  (x , \phi x), \qquad  x\in A, y\in B
\end{equation}
is a homomorphic conditional expectation of
$A\oplus B$ onto the graph of $\phi$.
Conversely, if
$\phi \colon A \to B$ is a function
and $\cexp$ given by (\ref{e:CexpOntoGraph}) is a homomorphic conditional expectation,
then $\phi$ is a $*$-homomorphism.
\end{example}

The projection on a direct sum of two $C^*$-algebras onto one of the summands is an example
of a homomorphic conditional expectation.  In particular, a split extension $E$ of a $C^*$-algebra $A$ by a $C^*$-algebra $B$ gives rise to homomorphic conditional expectations.

 \begin{example} In the theory of generalized inductive limits, due to Blackadar and Kirchberg (\cite[V.4.3]{Blackadar2006}),
 $NF$ algebras are not the same as strong $NF$ algebras (\cite[V.4.3.24.]{Blackadar2006}). Nevertheless, by  \cite[Corollary V.4.3.27]{Blackadar2006}, any $NF$ algebra $A$  is the range of a homomorphic conditional expectation defined on any split essential extension $B$ of $A$, which is in fact a strong $NF$ algebra. In this corollary, $A$ is called a retraction of $B$, which  partially motivated our use of the term {\it retraction}.
\end{example}

We will establish the following
characterization of homomorphic conditional expectations
in terms of operator norm and the Kadison-Schwarz inequality.
Recall that the Kadison-Schwarz inequality shows that
any conditional expectation  $\cexp \colon A \to A$
defined on a $C^*$-algebra $A$ satisfies
$\cexp(x)^* \cexp(x) \leq \cexp(x^* x)$
and consequently
$\norm{\cexp(x)}^2 \leq  \norm{\cexp(x^*x)}$
for every $x$ in $A$.

\begin{theorem}
\label{t:CharacterizationHomomorphicExpectations}
Let $A$ be a $C^*$-algebra
and let $\cexp \colon A \to A$ be a conditional expectation.
Then $\cexp$ is homomorphic if and only if
\begin{equation}
\norm{\cexp(x)}^2  =  \norm{\cexp(x^*x)} \quad \mbox{for every $x$ in $A$.}
\end{equation}
\end{theorem}

In the proof we will make use of the fact that  a 
closed Jordan ideal (defined in the proof of Lemma~\ref{l:|Cexp(x)|^2=|Cexp(x*x)|<=>KernelJordanIdeal}) in a $C^*$-algebras~$A$
is a two-sided ideal of $A$.  (\cite[Theorem~5.3]{CivinYood1965},
also see~\cite[Remark p.188]{BartonTimoney1986})
and the observations made in
Lemmas~\ref{l:CexpHomomorphic<=>KernelIdeal}
and ~\ref{l:|Cexp(x)|^2=|Cexp(x*x)|<=>KernelJordanIdeal}.

\begin{lemma}
\label{l:CexpHomomorphic<=>KernelIdeal}
A conditional expectation $\cexp \colon A \to A$
defined on a $C^*$-algebra $A$ is homomorphic if and only if
the kernel of $\cexp$ is an ideal in  $A$.
\end{lemma}
\begin{proof}
This is a straightforward consequence of conditional expectation properties.
\end{proof}

\begin{lemma}
\label{l:|Cexp(x)|^2=|Cexp(x*x)|<=>KernelJordanIdeal}
Let $A$ be a $C^*$-algebra.
If $\cexp \colon A \to A$ is a conditional expectation satisfying
$\norm{\cexp(x)}^2  =  \norm{\cexp(x^*x)}$ for all $x \in A$,
then the kernel of $\cexp$ is a closed Jordan $*$-ideal in $A$.
\end{lemma}

\begin{proof}
We use $M$ to denote the kernel of $\cexp$.
It is clear that $M$ is a closed linear subspace of $A$
which is also closed under the $*$-operation.
We need only to prove that $M$ is a Jordan ideal in the sense that
if $x\in A$ and $y\in M$,
then the Jordan product $x \cdot y = \frac{1}{2}(xy + yx)$ is in $M$.
The proof of this fact will proceed through several steps.

First,
if $y \in M$, then $y^*y \in M$
by the assumption $\norm{\cexp(y)}^2 = \norm{\cexp(y^*y)}$.
In~particular, $y^2 \in M$ for all self-adjoint elements $y \in M$.

Second,
if $y, z$ are self-adjoint elements of $M$,
then  by the preceding paragraph,
both $(y+z)^2$ and $(y-z)^2$ are in $M$, and one has
$$ y \cdot z = [(y+z)^2  -  (y-z)^2]/4. $$
It follows that $y \cdot z \in M$,
whenever $y, z$ are self-adjoint elements of $M$.

Third,
if $y, z$ are arbitrary elements of $M$,
write $y = y_1 + iy_2$ and $z = z_1 + iz_2$
with $y_i = y_i^*$ and $z_i = z_i^*$ in $M$,
and split the Jordan product $y \cdot z$ into real and imaginary parts as
$$y \cdot z = y_1 \cdot z_1 - y_2 \cdot z_2 + i(   y_1 \cdot z_2 + y_2 \cdot z_1).$$
By the preceding paragraph,
each of the four terms $y_i \cdot z_j$  appearing above is in $M$,
thus  $y \cdot z\in M$.
At this stage, we may conclude that $M$ is closed under the Jordan product
and we may indicate this by writing  $M \cdot M \subseteq M$.

Fourth,
since $M$ is invariant under both left and right multiplication
by elements of the range of $\cexp$, which we denote by $\cexp(A)$,
it follows that $\cexp(A) \cdot M \subseteq M$.
That is,
the Jordan product $\cexp(x) \cdot y$ is in $M$
for all $x\in A$ and all $y\in M$.

Finally,
if $x\in A$ and $y\in M$,
then the Jordan product
$$x \cdot y  =  \cexp(x) \cdot y   +   (x - \cexp(x)) \cdot y $$
is in $M$ because, by what we have proved,
$\cexp(x) \cdot y \in \cexp(A) \cdot M \subseteq M$
and $(x - \cexp(x)) \cdot y \in M \cdot M \subseteq M$.
Thus $M$ is a Jordan ideal (and a Banach $*$-subspace of $A$).
\end{proof}

We now turn to the proof of Theorem~\ref{t:CharacterizationHomomorphicExpectations}.

\begin{proof}[Proof of Theorem~\ref{t:CharacterizationHomomorphicExpectations}]
Let $\cexp \colon A \to A$ be a conditional expectation satisfying
$\norm{\cexp(x)}^2  =  \norm{\cexp(x^*x)}$  for every $x$ in $A$.
Then by Lemma~\ref{l:|Cexp(x)|^2=|Cexp(x*x)|<=>KernelJordanIdeal}
the kernel of $\cexp$ is a closed Jordan $*$-ideal in $A$,
hence a two-sided ideal.
It follows that $\cexp$ is homomorphic
by the observation made in Lemma~\ref{l:CexpHomomorphic<=>KernelIdeal}.
\end{proof}

We have already mentioned that if $p$ is a central projection in a $C^*$-algebra $A$, then the map $\cexp_p \colon A \to A$
defined by $\cexp_p(x) = px$,  for all $x\in A$, is a homomorphic conditional expectation.  We prove the converse in Proposition~\ref{prop:0129171}.

\begin{lemma}\label{lem:0126171}
Let $e$ be a projection in a von Neumann algebra $A$, and suppose $\cexp_e(x)=exe$ is a homomorphism.

(i) If $1-e$ is subequivalent to $e$, then $e=1$.

(ii) If $e$ is subequivalent to $1-e$, then $e=0$.
\end{lemma}
\begin{proof}
Since $\cexp_e$ is a homomorphism, we have $exeye=exye$ for every $x,y\in A$.  If $1-e$ is subequivalent to $e$, then by definition, there exists $u\in A$ satisfying $uu^*=1-e$ and $u^*u=h\le e$.  Then $$h=ehe=eu^*ue=eu^*eue=eu^*uu^*eue=eu^*(1-e)eue=0,$$ which proves (i).

If $e$ is subequivalent to $1-e$, there exists $u\in A$ satisfy $uu^*=e$ and $u^*u=h\le 1-e$.
Then $$e=euu^*e=eueu^*e=euu^*ueu^*e=euheu^*e=euh(1-e)eu^*e=0,$$ which proves (ii).
\end{proof}
\begin{proposition}\label{prop:0129171}
If $e$ is a projection in a $C^*$-algebra $A$, and $\cexp_e(x)=exe$ is a homomorphism, then $e$ belongs to the center of $A$.
\end{proposition}
\begin{proof}
By passing to the second dual, it suffices to assume that $A$ is a von Neumann algebra.  Apply the comparability theorem (\cite[III.1.1.10]{Blackadar2006}) to the projections $e$ and $1-e$ to obtain a central projection $z$ such that
$ze$ is subequivalent to $z(1-e)$ and $(1-z)(1-e)$ is subequivalent to $(1-z)e$.  With $A=Az\oplus A(1-z)$ we have $\cexp_e=\cexp_{ez}\oplus \cexp_{e(1-z)}$.  Then by Lemma~\ref{lem:0126171}, $ez=0$ and $e(1-z)=1-z$, so that $e=ez+e(1-z)=1-z$ is in the center of $A$.
\end{proof}

\begin{corollary}\label{cor:0126171}
If $e$ is a projection in a $C^*$-algebra $A$, and $\|exe\|=\|ex\|$ for every $x\in A$, then $e$ belongs to the center of $A$.
\end{corollary}
\begin{proof}
By Theorem~\ref{t:CharacterizationHomomorphicExpectations}, the assumption $\|exe\|=\|ex\|$ for every $x\in A$ implies that $\cexp_e$ is a homomorphism.
\end{proof}
As pointed out to us by Matt Neal, Corollary~\ref{cor:0126171} also follows from \cite[Lemmas 1.5 and 1.6]{FriRusJRAM85}. An elegant elementary proof of \cite[Lemma 1.5]{FriRusJRAM85} appears in \cite{PeraltaEM15}.  Another topological characterization of central projections is given in
\cite{Kato76}, namely a projection in a von Neumann algebra is central if and only if it is an isolated point in the set of projections with the norm topology.

\begin{remark}\label{rem:0602171}
The authors have recently learned  that there is an alternative
argument that proves
Theorem~\ref{t:CharacterizationHomomorphicExpectations}:
Since $x := a - \cexp(a)$ belongs to $\ker \cexp$,
$
|| \cexp(x^*x) || = || \cexp(x) ||^2 = 0 
$
immediately implies that $a$ belongs to
the multiplicative domain of $\cexp$ (\cite[Theorem 3.1]{Choi74}). This argument can be applied to prove
two other results (see Propositions~\ref{prop:0603171} and \ref{prop:0603172}).
However, the method presented in
our proof of Theorem~\ref{t:CharacterizationHomomorphicExpectations}
can be used to deduce a similar operator norm characterization of
multiplicative conditional expectations in the context of ternary rings of operators 
and Jordan triple systems (where the concept of
multiplicative domain is not applicable).  For example, see
Proposition~\ref{thm:1208161}.
\end{remark}

The two results which follow, and the tools used in their proofs, are valid for abstract JB$^*$-triples, for which a reference is the monograph \cite[Definition 2.5.25]{chubook}.   The principal example of a JB$^*$-triple is a JC$^*$-triple, that is, a norm closed subspace $A$ of a C$^*$-algebra which is closed under the symmetrized triple product $\{xyz\}_A:=(xy^*z+zy^*x)/2$.  We therefore phrase these
two results in this context.

A {\it triple homomorphism} is a linear mapping $T:A\rightarrow B$ between two JC$^*$-triples which preserves the triple product: $T\{xyz\}_A=\{Tx,Ty,Tz\}_B$.  A {\it triple ideal} is a subspace $I$ of a JC$^*$-triple $A$ satisfying $\{IAA\}_A+\{AIA\}_A\subset I$.

Let $A$ be a JC$^*$-triple, with triple product denoted $\{abc\}_A$ (or just  $\{abc\}$) and let $P:A\rightarrow A$ be a nonzero contractive
projection: $P^2=P$, $\|P\|=1$.  We have the ``conditional expectation'' formulas (\cite[Corollary 1]{FriRusPJM})
\begin{equation}\label{eq:1208164}
P\{x,Py,Pz\}=P\{Px,Py,Pz\}=P\{Px,y,Pz\}\hbox{ for all }x,y,z\in A.
\end{equation}
We recall (\cite[Theorem 2]{FriRusJFA}, \cite[Theorem 3.3.1]{chubook}) that $P(A)$ is isometric to a JC$^*$-triple under the norm of $A$ and the  triple product 
\begin{equation}\label{eq:0602172}
\{Px,Py,Pz\}_{P(A)}:=P(\{Px,Py,Pz\}_A).
\end{equation}

\begin{lemma}
\label{lem:1208161}
A contractive projection $P \colon A \to A$
defined on a JC$^*$-triple $A$ is a triple homomorphism of $A$ into $P(A)$, that is,  for all $a,b,c\in A$, 
\begin{equation}\label{eq:1208161}
P\{abc\}_A=\{Pa,Pb,Pc\}_{P(A)},
\end{equation}  if and only if
the kernel of $P$ is a triple ideal in  $A$.
\end{lemma}
\begin{proof}
Assume (\ref{eq:1208161}), let $a\in \ker P$, and let $b,c\in A$. Then $P\{abc\}_A=\{Pa,Pb,Pc\}_{P(A)}=P\{Pa,Pb,Pc\}_{A}=0$, and similarly, $P\{bac\}_A=0$.

Conversely, suppose $\ker P$ is an ideal.  For $x\in A$, with $x=Px+P'x$, where
$P'=I-P$, we have  (noting that $\{Px,Px,P'x\}=\{P'x,Px,Px\}$)
\[
\{xxx\}_A=\{Px+P'x,Px+P'x,Px+P'x\}=\{Px,Px,Px\}+y,
\]
where $y\in \ker P$.  Thus $P\{xxx\}_A=P\{Px,Px,Px\}_A=\{Px,Px,Px\}_{P(A)}$ and  by the polarization identity,
$$
\{xyz\}=\frac{1}{8}\sum_{\alpha^4=1,\beta^2=1}\alpha\beta\{x+\alpha y+\beta z,x+\alpha y+\beta z,x+\alpha y+\beta z\},
$$
$P$ is a triple homomorphism.
\end{proof}

\begin{proposition}
\label{thm:1208161}
Let $A$ be a JC$^*$-triple
and let $P \colon A \to A$ be a contractive projection.  Then $P$ is a triple homomorphism of $A$ onto $P(A)$ if and only if
$P$ satisfies
\begin{equation}\label{eq:1208162}
\{\ker P,\ker P,\ran P\}_A\subset \ker P,
\end{equation}
\begin{equation}\label{eq:1212161}
\{\ker P,\ran P,\ker P\}_A\subset \ker P,
\end{equation}
and
\begin{equation}\label{eq:1208163}
\norm{P(x)}^3  =  \norm{P\{xxx\}_A} \quad \hbox{for every }x\in A.
\end{equation}
\end{proposition}
\begin{proof}
If $P$ is a triple homomorphism, it is obvious that (\ref{eq:1208162}) and (\ref{eq:1212161}) hold, and if $x\in A$, then $$P\{xxx\}_A=\{Px,Px,Px\}_{P(A)},$$
so that 
$$\|P\{xxx\}_A\|=\|\{Px,Px,Px\}_{P(A)}\|=\|Px\|_{P(A)}^3=\|Px\|_A^3.$$

Conversely, assume (\ref{eq:1208162})-(\ref{eq:1208163}) hold.  We shall show that
$\ker P$ is an ideal, so that  Lemma~\ref{lem:1208161} is applicable.

For $x\in \ker P$ and $y,z\in A$, it is required to show
that $P\{xyz\}_A=0$ and $P\{yxz\}_A=0$.
Write $y=Py+P'y$, and $z=Pz+P'z$.  Then
\begin{eqnarray*}
\{xyz\}_A&=&\{P'x,Py+P'y,Pz+P'z\}\\
&=&\{P'x,Py,Pz\}+\{P'x,P'y,Pz\}\\
&+&\{P'x,Py,P'z\}+\{P'x,P'y,P'z\}.
\end{eqnarray*}
By (\ref{eq:1208163}), $\ker P$ is closed under $x\mapsto \{xxx\}_A$, so by the 
polarization identity, it is a subtriple of $A$, and therefore $P\{P'x,P'y,P'z\}=0$. By (\ref{eq:1208162}) and (\ref{eq:1212161}), 
$  P(\{P'x,P'y,Pz\}+\{P'x,Py,P'z\})=0$.  By  (\ref{eq:1208164}), $P\{P'x,Py,Pz\}=0$.  Thus $P\{xyz\}_A=0$ and a similar proof shows  $P\{yxz\}_A=0$. 
 \end{proof}

As noted in Remark~\ref{rem:0602171}, the technique mentioned there  can be used to show the following two results, which are responses to a question posed to the authors independently by C. Akemann and by  the referee.

A {\it JC$^*$-algebra} is a norm closed subspace $A$ of a C$^*$-algebra which is closed under the Jordan  product $x\circ y:=(xy+yx)/2$ and the involution.  A {\it Jordan homomorphism} is a linear mapping $T:A\rightarrow B$ between two JC$^*$-algebras which preserves the Jordan product: $T(x\circ y)=Tx\circ Ty$,  equivalently, $T(a^2)=T(a)^2$ for all $a=a^*$.

\begin{proposition}\label{prop:0603171}
Let $A$ be a $C^*$-algebra
and let $\cexp \colon A \to A$ be a conditional expectation.
Then $\cexp$ is a Jordan homomorphism if and only if
\begin{equation}\label{eq:0603171}
\norm{\cexp(x)}^2  =  \norm{\cexp(x^2)} \quad \mbox{for every $x=x^*$ in $A$.}
\end{equation}
\end{proposition}
\begin{proof}
If $\cexp$ is a Jordan homomorphism, then $\cexp(a^2)=\cexp(a)^2$ so (\ref{eq:0603171}) holds.
Conversely, if $a=a^*\in A$, then $x=x^*=a-\cexp(a)\in \ker\cexp$, and
\[
0=\cexp(x^2)=\cexp(a^2-a\cexp(a)-\cexp(a)a+\cexp(a)^2)
=\cexp(a^2)-\cexp(a)^2,
\]
so $\cexp$ is a Jordan homomorphism.
\end{proof}

Let $A$ be a unital JC$^*$-algebra, with Jordan product  denoted $a\circ b$, and let $P:A\rightarrow A$ be a nonzero positive unital
projection.  The conditional expectation formulas (\ref{eq:1208164}) reduce to
\begin{equation}\label{eq:0602171}
P(x\circ Py)=P(Px\circ Py),
\end{equation}
and by (\ref{eq:0602172}), $P(A)$ is isometric to a JC$^*$-algebra under the norm of $A$ and the  Jordan product $(a,b)\mapsto a*b:=P(a\circ b)$, for $a,b\in P(A)$ (see \cite[Theorem 1.4]{EffSto79} for
the original proof of the latter statement and \cite[Lemma 1.1]{EffSto79} for the original proof of (\ref{eq:0602171})).  Note that $P(a)*P(a)=P(P(a)^2)$.

\begin{proposition}\label{prop:0603172}
Let $A$ be a unital $JC^*$-algebra
and let $P:A \to A$ be a positive unital projection.
Then $P$ is a Jordan homomorphism, that is, $P(a^2)=P(P(a)^2)$ if and only if
\begin{equation}\label{eq:0603172}
\norm{Px}^2  =  \norm{P(x^2)} \quad \mbox{for every $x=x^*$ in $A$.}
\end{equation}
\end{proposition}
\begin{proof}
If $P$ is a Jordan homomorphism, 
so that $P(a^2)=P(P(a)^2)$, then $\|P(a^2)\|\le\|P(a)^2\|=\|P(a)\|^2$.  However, since 
$P$ is positive, $P(a^2)\ge P(a)^2$ (\cite[Theorem 1.2]{RobYou82}), so that (\ref{eq:0603172}) holds.

Conversely, if $a=a^*\in A$, then $x=x^*=a-\cexp(a)\in \ker\cexp$, and
\begin{eqnarray*}
0&=&P(x^2)=P(a^2-P(a)a-aP(a)+P(a)^2)\\
&=&P(a^2)-2P(P(a)\circ a)+P(P(a)^2)\\
&=&P(a^2)-P(P(a)^2)\hbox{ (by (\ref{eq:0602171}))},
\end{eqnarray*}
so $P$ is a Jordan homomorphism.
\end{proof}

\section{Homomorphic conditional expectations on $C_0(X)$}
\label{s:HomomorphicConditionalExpectationsOnC}

This section is based on~\cite[5.1]{Pluta2013}.
We  discuss the relationship between
homomorphic conditional expectations on commutative $C^*$-algebras $C_0(X)$
and retractions on $X$, for compact and locally compact Hausdorff spaces $X$.
When we deal specifically with a compact Hausdorff space we usually use $K$ in place of $X$.

If $K$ is a compact Hausdorff space,
we use $C(K)$ to denote the unital $C^*$-algebra
(with pointwise operations and the supremum norm)
of all complex-valued continuous functions~on~$K$.
If $X$ is a locally compact Hausdorff space,
we use $C_0(X)$ to denote the $C^*$-algebra
of all complex-valued continuous functions on $X$ which vanish at infinity.
If $K$ is compact, then $C_0(K) = C(K)$.

\begin{example}
Retractions  $\tau \colon K \to K$ on a (locally) compact Hausdorff  space $K$
give rise to homomorphic conditional expectations $\cexp_\tau \colon C(K) \to C(K)$ via $\cexp_\tau(f) = f \circ \tau$.
But there are expectations on $C(K)$ which do not come from any retraction of $K$
(those expectations are not homomorphic).
For instance,
let $K = \{ e^{i \theta} : 0 \leq \theta \leq 2 \pi\}$ and
define
\[
\cexp \colon C(K) \to C(K) \mbox{ by }  \cexp(f)(\zeta) = \frac{f (\zeta) +
f(-\zeta)}2.
\]
Then $\cexp$ is a   (not homomorphic)  conditional expectation  on $C(K)$
and there is no retraction  $\tau \colon K \to K$  with  $\cexp  = \cexp_\tau$;
see~\cite[Proposition 5.1.6]{Pluta2013}.
\end{example}

\begin{theorem}
\label{t:Retraction=HomomorphicExpectation}
Let $K$ be a compact Hausdorff space.
If $\tau\colon K \to K$ is a retraction (i.e., a continuous function  with $\tau \circ \tau = \tau$),
then the map
\begin{equation}
\label{e:CexpTau}
\cexp_\tau \colon C(K) \to C(K), \qquad \cexp_\tau(f) = f\circ \tau,
\qquad \mbox{for every $f$ in $C(K)$}
\end{equation}
is a unital homomorphic conditional expectation.

Conversely, if $\cexp \colon C(K) \to C(K)$ is  a unital  homomorphic conditional expectation,
then there is a  retraction $\tau\colon K \to K$ such that $\cexp = \cexp_\tau$,
where $\cexp_\tau$  is given by formula~(\ref{e:CexpTau}).
\end{theorem}

\begin{proof}
The first implication is a straightforward verification.
For the second implication,
let $\cexp \colon C(K) \to C(K)$ be a~unital homomorphic conditional expectation. Then the kernel of $\cexp$,  which  will be denoted by  $\ker \cexp$,
is a closed ideal and hence there is a closed set  $K_1 \subseteq K$
such that
$\ker \cexp = \set{f \in C(K) \,:\,   f\mid_{K_1}  =  0}$;
(see, for example, ~\cite[Theorem 4.2.4]{Rickart60} or \cite[Theorem 85]{Stone1937}).
If we let $S$ denote the range of $\cexp$,
then $S$ is a closed subalgebra of $C(K)$ (containing the constants)
and $\cexp$ induces an algebra isomorphism
$$\tilde{\cexp} \colon  C(K)/\ker\cexp \to S,  \qquad  \tilde{\cexp}(f+ \ker\cexp) = \cexp(f),
\qquad \mbox{for every $f$ in $C(K)$}.$$
We also have an isomorphism
$$\pi\colon  C(K)/\ker\cexp \to C(K_1), \qquad  \pi(f+ \ker\cexp) =  f\mid_{K_1},
\qquad \mbox{for every $f$ in $C(K)$};$$
(\cite[Theorem 4.2.4]{Rickart60}, or  \cite[Theorem 85]{Stone1937}).
Now $\tilde{\cexp} \circ \pi^{-1} \colon C(K_1) \to S  \subseteq C(K)$  is a unital algebra homomorphism
and so there exists a continuous function   $\phi \colon K \to K_1$
such that $(\tilde{\cexp} \circ \pi^{-1})(h) =  h \circ \phi$,
for $h\in C(K_1)$;  see  \cite[II.2.2.5]{Blackadar2006}.
Let $\tau \colon K \to K$ be given by $\tau(t) = \phi(t)$, for $t\in K$,
so that $\tau$ has the same values as $\phi$
but with a different co-domain.
Note that $\tau$ is continuous (since $\phi$ is).
We claim that $\cexp(f) = f \circ \tau$  for all  $f \in  C(K)$.
Indeed, if  $f \in  C(K)$, then $\pi(f + \ker\cexp) =  f\mid_{K_1}$  thus
$ f + \ker\cexp =  \pi^{-1}(f\mid_{K_1})$,
and this implies that  $(\tilde{\cexp} \circ \pi^{-1})(f\mid_{K_1}) = \cexp(f)$.
But also $(\tilde{\cexp} \circ \pi^{-1})(f\mid_{K_1}) =  (f\mid_{K_1}) \circ \phi = f \circ \tau$.
Hence we have $\cexp(f) = f \circ \tau$  for all  $f \in  C(K)$,  as claimed.
Since $\cexp(\cexp(f)) = \cexp(f)$  we must have $f \circ \tau \circ \tau = f \circ \tau$
for each function  $f \in  C(K)$.
Since the functions in $C(K)$ separate the points of $K$, it follows that $\tau \circ \tau = \tau$
so that $\tau$ is a retraction.
\end{proof}

\begin{corollary}
\label{c:E-gives-retraction-compact-case}
Let $K$ be a compact Hausdorff space and let $\cexp \colon C(K) \to C(K)$
be a  homomorphic conditional expectation.
Then there is a
clopen  set $L \subseteq K$
and a retraction  $\tau \colon L \to L$ such that $\cexp$ is given by
\[
\cexp(f)(t) =
\begin{cases}
f (\tau (t)) & \mbox{ if } t \in L\\
0 & \mbox{ for } t \in K \setminus L
\end{cases}
\]
for $f \in C(K)$, $t \in K$.
\end{corollary}

\begin{proof}
Let $1_K$ denote the constant function  1  in $C(K)$.
Then
$\cexp(1_K)^2  =  \cexp(1_K)$ and so there is
$L = \{t \in K : \cexp(1_K)(t) = 1 \}$ so that $\cexp(1_K) = 1_L$
(the characteristic function of $L$).
Moreover, since $1_L \in C(K)$, $L \subseteq K$ is a clopen subset.
Since $\cexp(1_K - 1_L) = \cexp(1_K) - \cexp(1_L) = \cexp(1_K)-\cexp(\cexp(1_K))=0$,
we have that  $1_K - 1_L \in \ker \cexp$ (which is an ideal).
So if
$f \in C(K)$, then \begin{equation}\label{eq:0131171}
\cexp(f) = \cexp( 1_L f + (1_K - 1_L)f) =
\cexp( 1_L f) = 1_L \cexp(1_L f).
\end{equation}

Identifying $C(L)$ with $\{ f \in C(K) : f = 1_L f \}$
via $g \in C(L) \mapsto \tilde{g} \in C(K)$
(where $\tilde{g}\mid_L = g$ and $\tilde{g}(t) = 0$ for $t \in K \setminus L$),
we see that
$\cexp$ induces a homomorphic unital  conditional expectation
$\cexp_L \colon C(L) \to C(L)$ by $\cexp_L ( g) = \cexp(\tilde{g}) \mid_L$. The result follows by
applying Theorem~\ref{t:Retraction=HomomorphicExpectation} to $\cexp_L$ and using (\ref{eq:0131171}).
\end{proof}

Let $X$ be a locally compact Hausdorff space and
$X^* = X \cup \{\omega\}$ the one point compactification. We use this notation here
even when $X$ is already compact, in which case $\{ \omega \}$ is open (and closed) in $X^*$.
Subsets $U$ of $X^*$ are open if $U \cap X$ is open in $X$
and if $\omega \in U$ we insist that $X^* \setminus U$ be a compact subset of $X$.

We consider $C_0(X)$ as embedded in $C(X^*)$ via
\[
f \mapsto \tilde{f} \colon C_0(X) \to C(X^*),
\]
where
\[
\tilde{f}(t) = \begin{cases}
f(t) & \mbox{ if } t \in X\\
0 & \mbox{ for } t =\omega.
\end{cases}
\]
Note that this identifies $C_0(X)$ with $\{ g \in C(X^*) :
g(\omega) = 0 \}$  (the maximal ideal of $C(X^*)$ consisting of
functions which take the value zero  at $\omega$)
and $f \mapsto \tilde{f}$ is a *-algebra isomorphism
onto its range.

If $\tau \colon X^* \to X^*$ is a retraction  such that $\tau(\omega) =
\omega$, then we can define a conditional expectation
$
\cexp_{\tau, *} \colon C_0(X) \to C_0(X)
$
by
$
\cexp_{\tau, *}(f) = (\tilde{f} \circ \tau) |_X
$.

\begin{corollary}
\label{c:yes-locally-comp-case}
If $X$ is a locally compact Hausdorff space and $\cexp \colon C_0(X) \to C_0(X)$
is a homomorphic conditional expectation,
then there is a retraction  $\tau \colon X^* \to X^*$  ($X^* = X \cup \set{\omega}$)  with
$\tau(\omega) = \omega$ such that $\cexp = \cexp_{\tau, *} $.
\end{corollary}

\begin{proof}
First consider the case when $X$ is compact.
We apply Corollary~\ref{c:E-gives-retraction-compact-case} above
to get $L \subseteq X$ compact and clopen and $\tau \colon L \to L$ a retraction  with
\[
\cexp(f)(t) = \begin{cases}
f(\tau(t)) & \mbox{ if } t \in L\\
0 & \mbox{ if }t \in X \setminus L.
\end{cases}
\]
Define a retraction  $\rho \colon X^* \to X^*$ by $\rho(t) = \tau(t)$ for
$t \in L$ and $\rho(t) = \omega$ for
$t \in (X \setminus L) \cup \{\omega\}$. Since
$L$ is clopen and so is $\{\omega\}$, $\rho$ is continuous. We can
verify that $\rho \circ \rho = \rho $ and $\cexp = \cexp_{\rho, *}$.

In the case that $X$ is not compact,
note that $C(X^*)$ is isomorphic as a *-algebra to the unitisation
$C_0(X)^{^\sharp}$,
where $C_0(X)^{^\sharp}$ is defined as in~\cite[Definition 1.3.3]{Dales2000}.
The isomorphism is
given by $g \mapsto  \phi(g) := (g|_X - g(\omega), g(\omega))$.
Indeed, if $\phi(g_1) = \phi(g_2)$, then $g_1(\omega) = g_2(\omega)$ and $g_1|_X = g_2|_X$,
thus $g_1 = g_2$.
On the other hand, if $(h, \alpha) \in C_0(X)^{^\sharp}$, then $\phi(g) = (h, \alpha)$, where

\[
g(x) =
\begin{cases}
h(x) + \alpha & \mbox{ if } x \in X\\
\alpha  & \mbox{ if } x = \omega.
\end{cases}
\]

Regard
$\cexp^{^\sharp} \colon C_0(X) \oplus \IC\to C_0(X) \oplus \IC$ as
$\cexp^{^\sharp}  \colon C(X^*) \to C(X^*)$,
where $\cexp^{^\sharp}(h, \alpha)  =  (\cexp(h), \alpha)$ for $h\in C_0(X)$, $\alpha \in \IC$.
Then $\cexp^{^\sharp}$ is an~algebra homomorphism and a~conditional expectation.

We apply Corollary~\ref{c:E-gives-retraction-compact-case}
 to get $L \subset X^*$ clopen
and $\tau \colon L \to L$ a retraction  so that

\begin{equation}
\label{e:Lam-exp-c2}
\cexp^{^\sharp}  (f)(t) =
\begin{cases}
f(\tau(t)) & \mbox{ if } t \in L\\
0 & \mbox{ if }t \in X^* \setminus L,
\end{cases}
\end{equation}
for $f\in C(X^*)$.
Since
$C_0(X)$ can be identified with the maximal ideal of $C(X^*)$ consisting of
functions which take the value zero  at $\omega$, i.e., $C_0(X) = \{ g \in C(X^*) : g(\omega) = 0\}$,
we have
$$\cexp^{^\sharp} (C_0(X)) = \cexp^{^\sharp} ( \{ g \in C(X^*) : g(\omega) = 0\})
\subset C_0(X),$$
and, therefore, if $\omega \in L$, then $\tau(\omega) = \omega$.
(Indeed, if $\omega \in L$ and
$\tau(\omega) = t \in X$, there is $f \in C_0(X)$ with $f(t) = 1$
and
 we would have a contradiction from $0 = \cexp^{^\sharp}(f)(\omega) =
\tilde{f}(\tau(\omega)) = f(t) \neq 0$).
If $\omega \notin L$, then  $\omega \in X^* \setminus L$, $L\subseteq X$ is compact,
and $\cexp^{^\sharp}$ is given by~(\ref{e:Lam-exp-c2}).

Thus we can extend $\tau$ to a retraction  $\rho \colon X^* \to X^*$
by $\rho(t) = \tau(t)$ for
$t \in L$ and $\rho(t) = \omega$ for $t \in X^* \setminus L$. Since
$L$ is clopen, $\rho$ is continuous, and we can
verify that $\rho \circ \rho = \rho $ and $\cexp = \cexp_{\rho, *}$.
\end{proof}

\noindent {\bf Acknowledgment.} This paper was begun by the first named author at  Trinity College Dublin in coordination with Richard Timoney, to whom he now expresses his thanks.



\begin{bibdiv}
\begin{biblist}


\bib{BartonTimoney1986}{article}{
   author={Thomas Barton and Richard M. Timoney},
   title={Weak*-continuity of Jordan triple products and its applications},
   journal={Math. Scand.},
   volume={59},
   number={2},
   date={1986},
   pages={177--191},
}

\bib{Blackadar2006}{book}{
   author={Blackadar, Bruce},
   title={Operator algebras. Theory of $C^*$-algebras and von Neumann algebras},
   series={Encyclopaedia of Mathematical Sciences},
   volume={122},
   publisher={Springer-Verlag},
   place={Berlin},
   date={2006},
   pages={xx+517},
  }

 


\bib{Choi74}{article}{
   author={Man Duen Choi},
   title={A Schwarz inequality for positive linear maps on C$^*$-algebras},
   journal={Illinois J. Math.},
   volume={18},
   date={1974},
   pages={565--574},
}

\bib{chubook}{book}{
   author={Chu, Cho-Ho},
   title={Jordan structures in geometry and analysis},
   series={Cambridge Tracts in Mathematics},
   volume={190},
   publisher={Cambridge University Press},
   place={Cambridge},
   date={2012},
   pages={x+261 },
   }



\bib{CivinYood1965}{article}{
   author={Paul Civin and Bertram Yood},
   title={Lie and Jordan structures in Banach algebras},
   journal={Pacific J. Math.},
   volume={15},
   number={3},
   date={1965},
   pages={775--797},
}

\bib{Dales2000}{book}{
    AUTHOR = {Dales, H. Garth},
     TITLE = {Banach algebras and automatic continuity},
    SERIES = {London Mathematical Society Monographs. New Series},
    VOLUME = {24},
      NOTE = {Oxford Science Publications},
 PUBLISHER = {The Clarendon Press Oxford University Press},
   ADDRESS = {New York},
      YEAR = {2000},
     PAGES = {xviii+907},
      ISBN = {0-19-850013-0},
   }
	
	

\bib{EffSto79}{article}{
   author={Effros, Edward G.},
   author={St\o rmer, Erling},
   title={Positive projections and Jordan structure in operator algebras},
   journal={Math. Scand.},
   volume={45},
   number={1},
   date={1979},
   pages={127--138},
}


\bib{FriRusPJM}{article}{
   author={Friedman, Yaakov},
   author={Russo, Bernard},
   title={Conditional expectation without order},
   journal={Pacific J. Math.},
   volume={115},
   number={2},
   date={1984},
   pages={351--360},
}






\bib{FriRusJFA}{article}{
   author={Friedman, Yaakov},
   author={Russo, Bernard},
   title={Solution of the contractive projection problem},
   journal={J. Funct. Anal.},
   volume={60},
   number={1},
   date={1985},
   pages={56--79},
}


\bib{FriRusJRAM85}{article}{
   author={Friedman, Yaakov},
   author={Russo, Bernard},
   title={Structure of the predual of a JBW*-triple},
   journal={J. Reine Angew. Math.},
   volume={356},
   date={1985},
   pages={67--89},
}




\bib{Kato76}{article}{
   author={Yoshinobu Kato},
   title={Some theorems on projections of von Neumann algebras},
   journal={Math. Japon},
   volume={21},
   number={4},
   date={1976},
   pages={367-370},
}


\bib{PeraltaEM15}{article}{
   author={Antonio M. Peralta},
   title={Positive definite hermitian mappings associated with tripotent elements},
   journal={Expo. Math.},
   volume={33},
   number={2},
   date={2015},
   pages={252--258},
}



\bib{Pluta2013}{book}{
   author={Pluta, Robert},
   title={Ranges of bimodule projections and conditional expectations},
   series={ },
   volume={ },
   note={ },
   publisher={Cambridge Scholars Publishing},
   place={Newcastle upon Tyne},
   date={2013},
   pages={ },
   isbn={ },
   isbn={ },
}

	
\bib{Rickart60}{book}{
    AUTHOR = {Rickart, Charles E.},
     TITLE = {General theory of {B}anach algebras},
    SERIES = {The University Series in Higher Mathematics},
 PUBLISHER = {D. van Nostrand Co., Inc., Princeton, N.J.-Toronto-London-New York},
      YEAR = {1960},
     PAGES = {xi+394},
   }

\bib{RobYou82}{article}{
   author={Robertson, A. Guyan},
   author={Youngson, Martin A.},
   title={Positive projections with contractive complements on Jordan algebras},
   journal={J. London Math. Soc.},
   volume={25},
   series={2},
   number={2}
   date={1982},
   pages={365--374},
}


\bib{Stone1937}{article}{
   author={Stone, Marshall H.},
   title={Applications of the theory of Boolean rings to general topology},
   journal={Trans. Amer. Math. Soc.},
   volume={41},
   number={}
   date={1937},
   pages={375--481},
}

\bib{Tomiyama1957}{article}{
   author={Tomiyama, Jun},
   title={On the projection of norm one in $W^*$-algebras},
   journal={Proc. Japan Acad.},
   volume={33},
   number={10},
   date={1957},
   pages={608--612},
}


\end{biblist}
\end{bibdiv}
\end{document}